\documentclass[12pt]{article}

\usepackage[margin=0.7in]{geometry}
\usepackage{url}
\usepackage[round]{natbib}   % omit 'round' option if you prefer square brackets
\usepackage{authblk}
\usepackage{color}
\usepackage[title]{appendix}

\usepackage{algorithm,algpseudocode}
\usepackage{mathrsfs}
\usepackage{amsmath}  
\usepackage{mathtools}
\usepackage{amssymb}
\usepackage{amsthm}

\newtheorem{theorem}{Theorem}
\newtheorem{corollary}{Corollary}

\newtheorem{definition}{Definition}

\newtheorem{assumption}{Assumption}

\DeclarePairedDelimiterX\Set[2]{\lbrace}{\rbrace}%
 { #1 \,\delimsize| \,\mathopen{} #2 }

\newcommand{\E}{\mathbb{E}}
\newcommand{\M}{\mathcal{M}(\Xi,\mathcal{F})}
\newcommand{\Pset}{\mathcal{P}(\Xi,\mathcal{F})}
\newcommand{\W}{\mathcal{W}}
\newcommand{\bs}[1]{\boldsymbol{#1}}
\newcommand{\x}{\boldsymbol{x}}
\newcommand{\xivec}{\boldsymbol{\xi}}

\begin{document}
%%%%%%%%%%%%%%%%
\title{Distributionally Robust Optimization with Decision Dependent Ambiguity Sets}

\author{Fengqiao Luo and Sanjay Mehrotra}
\affil{Department of Industrial Engineering and Management Science, Northwestern University, Evanston, IL}

\maketitle

\abstract{
We  study  decision  dependent  distributionally  robust  optimization models,  where  the  ambiguity  sets  of  probability  distributions can  depend  on the decision  variables.  These  models  arise  in  situations  with  endogenous  uncertainty.  The  developed  framework  includes  two-stage  decision  dependent  distributionally  robust  stochastic  programming  as  a  special  case.  Decision  dependent  generalizations of five types of ambiguity  sets  are  considered. 
These sets are based on bounds on moments, Wasserstein metric, $\phi$-divergence and Kolmogorov-Smirnov test.
For  the  finite  support  case,  we  use  linear,  conic  or  Lagrangian  duality  to  give  reformulations  of  the  models  with  a  finite  number  of  constraints.  
These reformulations allow solutions of such problems using  global  optimization  techniques. Certain  reformulations  give  rise  to  non-convex  semi-infinite  programs.  Techniques  from  global  optimization  and  semi-infinite  programming  can  be  used  to  solve  these  reformulations. 
}

%%%%%%%%%%%%%%%%%%%%%%%%%%%%%%%%%%%%%%%%%%%%%%%%%%%%%%%%%%%%%%%%%%%%%%

\section{Introduction}
\label{sec:d3ro_introd}
The uncertain characteristics of a system's performance often depend on its design decisions. 
This type of uncertainty is called endogenous uncertainty. For example in a newsvendor model  
product demand function may depend on its selling price \citep{hu2015}.   
Additional examples of decision problems with endogeneous uncertainty 
from finance, resource management, process design, and network design are given in Section~\ref{sec:liter-rev-DDRO}. 
The goal of this paper is to present decision dependent ambiguity frameworks to model problems involving endogenous uncertainty. 
The main contribution is in showing that the dualization of a certain inner problem continues to be applicable in this more general setting. 
This dualization has a unique advantage for the problems under consideration. 
It allows application of algorithms from nonlinear global optimization to solve the resulting reformulations.  

Specifically, we study the optimization problems in which the ambiguity set of distributions 
may depend on the decisions in the following modeling framework:
\begin{equation}
\label{opt:DRO}
\underset{\bs{x}\in X}{\textrm{min}}\; \left\{ f(\bs{x}) + \underset{P\in\mathcal{P}(\bs{x})}{\textrm{max}}\; \E_P[h(\bs{x},\bs{\xi})] \right\}. \tag{\textrm{D$^3$RO}}
\end{equation} 
Here $\bs{x}$ is the vector of decision variables with the feasible set $X\subseteq\mathbb{R}^n$, 
and $\bs{\xi}$ is the vector of uncertain model parameters, which is defined on a measurable space $(\Xi,\mathcal{F})$;
$\Xi$ is the support in $\mathbb{R}^d$, and $\mathcal{F}$ is a $\sigma$-algebra.
For a given $\x$, the ambiguity set $\mathcal{P}(\bs{x})$ of the unknown probability distribution depends on the decision variable $\x$,
and $\mathcal{P}(\x)\subseteq \mathcal{P}(\Xi,\mathcal{F})$, where $\mathcal{P}(\Xi,\mathcal{F})$ 
is the set of probability distributions defined on $(\Xi,\mathcal{F})$. 
The function $f(\bs{x})$ is the deterministic part of the objective with no uncertain parameters. Keeping this function  in \eqref{opt:DRO} allows
us to consider decision models involving two-stage decision making. 
We denote the inner problem $\underset{P\in\mathcal{P}(\bs{x})}{\textrm{max}}\; \E_P[h(\bs{x},\bs{\xi})]$
as \eqref{opt:DRO}-inner. Note that if $h(\bs{x},\xi)$ is a recourse function in a two-stage stochastic program, i.e., 
\begin{equation}\label{eqn:h_recourse}
h(\bs{x},\bs{\xi})=\;\underset{\bs{y}\in\mathbb{R}^q}{\textrm{min}}\;\; g(\bs{x},\bs{y},\bs{\xi}),\quad \textrm{ s.t. } 
\psi_i(\x,\bs{y},\bs{\xi})\ge 0 \quad \forall i\in[m],
\end{equation}
where $g(\bs{x},\bs{y},\bs{\xi})$ and $\psi_i(\x,\bs{y},\bs{\xi})$ are bounded and continuous functions of $\x$, $\bs{y}$ and $\bs{\xi}$,
then \eqref{opt:DRO} becomes a two-stage decision dependent distributionally robust stochastic program (TSD$^3$SP),
which is an important application of \eqref{opt:DRO}. We assume that the minimization problem in \eqref{eqn:h_recourse}
is feasible for any $\x\in X$ and $\bs{\xi}\in\Xi$, and $h(\x,\bs{\xi})$ is finite. In other words, we assume that
 (D$^3$SP) has complete recourse \citep{birge1997-introd-stoch-prog}.

The ambiguity set $\mathcal{P}(\x)$ can be constructed in many different ways. The reformulations given in this paper
are for the decision dependent generalizations of the most common types of ambiguity sets proposed in the distributionally
robust optimization literature. The current work on distributional robust optimization, which assumes that the distribution of uncertain parameters is decision independent, 
is reviewed in Section~\ref{sec:liter-rev-DRO}.  In Section~\ref{sec:d3ro_reform} we investigate the reformulation of \eqref{opt:DRO} for five 
different possible specifications of $\mathcal{P}(\bs{x})$: (i) the ambiguity sets defined by using component-wise moment 
inequalities and bounds on the scenario probabilities; 
(ii) ambiguity sets defined by using the mean vector and covariance matrix inequalities; (iii) ambiguity sets defined by using 
the Wasserstein metric; (iv) ambiguity sets defined using $\phi$-divergence; and (v) the ambiguity sets defined using  the
multi-variate Kolmogorov-Smirnov test. The reformulations are given in Sections~\ref{sec:reform-simple-bound} 
to \ref{sec:reform-KS}, respectively. The basic concept used in arriving at the reformulations of \eqref{opt:DRO}-inner 
is to use linear programming duality or conic duality as needed in the specific settings. Lagrangian duality
is used for the situations where considering the saddle point problem appears more suitable.

We note that the computational complexity of the reformulated problem is not the
main motivation of this paper. Our goal is towards studying the modeling frameworks that more realistically represent
the underlying phenomenon. Here we also do not focus on developing any new (possibly more efficient) 
algorithms, as that is left for future studies. In general, we refer to the global optimization techniques for solving 
the non-convex optimization problems resulting from our reformulations \citep{2014-adv-glob-opt}.
Moreover, to simplify the presentation, 
we consider the finite support case in the main text. In some cases the results are also possible for ambiguity sets 
allowing for continuous support. In these cases 
the semi-infinite programming reformulation of the corresponding models allow the use of a cutting surface,
and possible other algorithms, to solve the problems.
These results and a well-known cutting surface algorithm for global semi-infinite programming is given in the appendix. 
The construction of the decision dependent parameters appearing
in the specification of the ambiguity set is discussed briefly when making the concluding remarks.

\section{Literature Review}
\label{sec:liter_rev} 
 We first review prior work on optimization models and methods for problems where decision influences the problem parameters.
We will subsequently provide a literature review of recent developments in the area of distributionally robust optimization
that is relevant to the current paper.

\subsection{Literature review on optimization with decision dependent uncertainty}
\label{sec:liter-rev-DDRO}
The endogenous uncertainty has been considered in dynamic programming \citep{webster2012_approx-DP-frame-glob-clim-policy-ddu}, 
stochastic programming \citep{grossmann2006_stoch-prog-dec-uncert}, 
robust optimization \citep{poss2013_rob-comb-opt-var-budgt-uncert,nohadani2016_opt-dec-uncert},  
with applications in financial market modeling 
\citep{kurz1974_ks-model-treat-uncert-equil-thy,kurz1996_ration-belief-endog-uncert,kurz2001_endog-uncert-market-volatility}, 
resource management \citep{tsur2004_gdwater-threat-catastr-event}, stochastic traffic assignment \citep{shao2006_reliab-stoch-traff-assign-demand-uncert}, oil (natural gas) exploration \citep{jonsbraten-thesis1998_opt-mod-petrol-field-expl,grossmann2009_stoch-prog-plan-oil-infrast-dec-uncert,grossmann2004_stoch-prog-plan-gas-field-uncert-reserv}, and robust network design \citep{ahmed2000-phd_plan-uncert-stoch-MIP,viswanath2004_invest-stoch-netwk-min-exp-short-path}.

%\cite{kurz1974_ks-model-treat-uncert-equil-thy} first addressed the concept of endogenous uncertainty in behavioral analysis of financial markets.
%In the model of belief propagation studied by \cite{kurz1996_ration-belief-endog-uncert,kurz2001_endog-uncert-market-volatility},
%the price of an asset has endogenous uncertainty since it is the equilibrium of market participants' beliefs conditional on the social states 
%(including historic prices, overall market performance, etc.), and hence the price varies under different social states.  

In the framework of stochastic optimization, the endogenous uncertainty affects the underlying probability distribution and the scenario tree.
%The work of \cite{pflug1990_online-opt-simult-markov-proc} appears as the earliest reference in the optimization literature that initiated 
%the research on decision making with endogenous uncertainty where 
%the underlying stochastic process depends on the optimization decisions. He studied the problem of optimizing the steady-state performance of a Markov process where the transition function (or matrix) is controlled by the decision variable. He provided a stochastic gradient method in which the estimator of the control variable converges to the optimal control under some regularity conditions.
\citet{jonsbraten1998_stoch-prog-dec-dept-rand-elemt} first studied stochastic programming problems with decision dependent scenario distributions, where the distribution is indexed by a Boolean vector. They provided an implicit enumeration algorithm for solving these problems based on a branch-and-bound scheme. This model and the proposed branch-and-bound method was applied to an optimal selection and sequencing of oil well exploration problem under reservoir capacity uncertainty \citep{jonsbraten-thesis1998_opt-mod-petrol-field-expl}. 
\citet{ahmed2000-phd_plan-uncert-stoch-MIP} investigated a class of single stage stochastic programs with discrete candidate probability distributions
that are based on Luce's choice axiom~\citep{luce1977-choice-axiom}. 
The decision affects utility functions of the choices, and hence the probability distribution \citep{mehrotra2015-rob-portf-opt-utility}. 
%\textcolor{red}{For the following references, check if they involve utility functions.}
These types of problems arise from network design and server selection applications~\citep{ahmed2000-phd_plan-uncert-stoch-MIP}. It is shown that stochastic programs of this class can be reformulated as 0-1 hyperbolic programs. \citet{viswanath2004_invest-stoch-netwk-min-exp-short-path} investigated a two-stage shortest path problem in a stochastic network which arises from disaster relief services. Here the first stage investment decisions can reduce the failure probability of links in the network, and a shortest path is identified based on the post-event network. \citet{held2005_heuristic-multi-stage-stoch-netwk-interd} developed a heuristic algorithm to solve a two-stage stochastic network interdiction problem, where the interdictor is the first-stage decision maker whose objective is to maximize the probability that the minimum path length exceeds a certain value after the interdiction. The interdictor's decision changes the network topology and the uncertainty description. The structure of single-stage stochastic programming with a decision dependent probability distribution was also studied by  
\citet{dupacova2006_opt-exog-endog-uncert}. 
\citet{lee2012_newsvd-model-ddu} investigated a newsvendor model under decision dependent 
uncertainty, where sequential decisions are made after a re-estimation of the demand distribution. 
They provide conditions under which the estimation and decision process converges.

\citet{grossmann2004_stoch-prog-plan-gas-field-uncert-reserv,grossmann2006_stoch-prog-dec-uncert} developed 
a disjunctive programming reformulation for multistage decision dependent stochastic programs. They investigated this problem with finitely many scenarios, 
where exogenous and endogenous uncertain parameters are involved. In their models, endogenous parameters 
are resolved after the operational decisions are made (e.g., a facility is installed or an investment is made).  
A branch-and-bound algorithm is developed to solve the disjunctive program by branching on the logic variables
involved in the disjunctive clauses \citep{grossmann2006_stoch-prog-dec-uncert,goel2006_BB-alg-gas-field-reserv-uncert},
and a lower bound is obtained at each node by solving a Lagrangian dual sub-problem \citep{goel2005_lag-dual-BB-lin-stoch-prog-ddu}.
More solution strategies for the disjunctive program are given in \citep{gupta2011_sol-strat-multi-stoch-prog-endog-uncert}.
This framework is applied to model and solve the offshore oil or gas field infrastructure multi-stage planning problem with uncertainty 
in estimating parameters that are not immediately realized \citep{grossmann2009_stoch-prog-plan-oil-infrast-dec-uncert}.
The framework is also applied to optimize process network synthesis problems with yield uncertainty that can be reduced by investing in pilot plants 
\citep{tarhan2008_multi-sddd-oil-synthesis}. \citet{tarhan2013_comp-methd-nonconv-multistage-MINLP-dec-dept-uncert} 
developed a computational strategy that combines global optimization and outer-approximation to solve multistage nonlinear 
mixed-integer programs with decision dependent uncertainty.

Decision dependent uncertainty is also considered in the framework of robust optimization by letting the uncertainty set depend on the decision variables.
\citet{spacey2012_rob-software-partition-multi-instant} studied a problem of minimizing the run time of a computer program by assigning code segments to execution locations where the scheduling of code segment execution depends on the assignment.
In robust combinatorial optimization, decision dependent uncertainty set is used to ensure the same relative protection level of
all binary decision vectors \citep{poss2013_rob-comb-opt-var-budgt-uncert}. 
To model a robust task scheduling problem with uncertainty in the processing time, \citet{vujanic2016_rob-opt-schedules-affect-uncert-event} 
proposed a decision dependent uncertainty set as a Minkowski sum of some static sets such that the uncertain completion time interval
of a task can naturally depend on the starting time of the task.
\citet{hu2015} studied a newsvendor model where the product demand may depend on the selling price. 
Since the analytical relationship between the demand and the selling price is unknown, they construct a family
of decreasing and convex functions from historical data as the functional ambiguity set of the true demand function,
and solve the functionally robust optimization problem of this model using a univariate reformulation. 
\citet{nohadani2016_opt-dec-uncert} investigated robust linear programs with decision dependent budget-type uncertainty (RLP-DDU). 
They showed that this problem is NP-hard even in the case where the uncertainty set is a polyhedron and the decision dependence is affine.
RLP-DDU can be reformulated as a mixed integer linear program (MILP), if the decision variables affect uncertain variables by controlling the upper bounds of the uncertain variables. 
This concept is demonstrated in a robust shortest-path problem,
where the uncertainty is resolved progressively when approaching the destination.

\subsection{Literature review on distributionally robust optimization}
\label{sec:liter-rev-DRO}
Distributionally robust optimization is a generalization of the classical robust optimization framework. It treats uncertain parameters 
as random variables with an unknown probability distribution. In the DRO framework the unknown distribution is described by an ambiguity set of 
probability distributions. The DRO framework solves a min-max problem, 
and identifies an optimal solution by assuming that nature will pick a worst case probability distribution 
based on the decision maker's choice.

Current approaches to constructing the ambiguity set are based on moment inequalities that specify the set of candidate probability distributions, 
and statistical distances between a candidate distribution and a reference distribution (see \citep{bertsimas2010_models-minmax-stoch-LP-risk-aver,birge1987,delage2010_distr-rob-opt-mom-uncer, 
dupacova1987_minmax-stoch, mehrotra2015, book_stoch_prog_1995, shapiro2004_class-of-minimax-analysis-stoch-progs, 
shapiro2002_minimax-analysis-stoch-probs}). 
Specifically, \citet{bertsimas2010_models-minmax-stoch-LP-risk-aver} studied two-stage stochastic linear programs 
with fixed recourse and specified the ambiguity set using first and second moments of the uncertain parameters. 
They showed that this problem can be reformulated into a semidefinite program.
\citet{delage2010_distr-rob-opt-mom-uncer} studied DRO problems with uncertain parameters
in the objective, while the ambiguity set is defined by conic inequalities on the mean vector and the covariance matrix. 
They showed that this DRO model is polynomial-time solvable under the assumption 
that the objective is convex in the decision variables, and it is concave in the random parameters.
The DRO version of the least-square problem, which
does not satisfy the concavity assumption in \citep{delage2010_distr-rob-opt-mom-uncer},
was studied in \citep{mehrotra2014_model-algthm-distr-rob-least-square}. It was
shown to admit a SDP reformulation for this case.

Using statistical distances is another way to define the ambiguity set. A statistical distance measures the 
difference between two probability distributions, and hence the ambiguity set can be naturally defined as
a set of probability distributions that are within a certain distance from a reference distribution. 
Within various types of statistical distances, the Wasserstein metric is a useful choice in defining 
the ambiguity set due to its tractability. The study of DRO problems  
with ambiguity sets defined using the Wasserstein metric and the solution approach are developed in 
\citep{gao2016-dro-wass-distance,esfahani2015_distr-rob-wass-metric,s-abadeh2015_distr-rob-log-reg,luo2017-decomp-alg-distr-rob-opt}.

DRO problems with ambiguity sets defined using $\phi$-divergence 
are investigated in \citep{ben-tal2013_rob-opt-uncert-prob,calafiore2007_amb-risk-meas-opt-rob-portf,jiang2016_data-driven-chance-constr-stoch-program,love2016_phi-diverg-constr-amb-stoch-prog-data-driven-opt,wang2016_likehd-rob-opt-data-driven,yanikoglu2013_safe-approx-amb-constr-hist-data}. Specifically, \citet{calafiore2007_amb-risk-meas-opt-rob-portf} studied a robust portfolio selection problem
using KL divergence to characterize the ambiguity set. \citet{ben-tal2013_rob-opt-uncert-prob} showed that the robust counterpart of linear optimization problems with uncertainty set defined by $\phi$-divergence are tractable for most choices of the function $\phi$. 
\citet{jiang2016_data-driven-chance-constr-stoch-program} investigated the distributionally robust
chance constraint models where the ambiguity set is defined using $\phi$-divergence. 
They showed that this type of chance constraint is equivalent to the classical
chance constraint with a perturbed risk level, and this risk level can be evaluated using a line search algorithm. 
\citet{yanikoglu2013_safe-approx-amb-constr-hist-data} proposed a method to construct a confidence region of an unknown random vector
using its samples. The method is based on partitioning the sample space into cells, and approximating the continuous unknown
probability distribution with counts in each cell. A subset of cells is selected to form the confidence region, 
and the cell selection process is formulated as a convex optimization problem with 
a constraint that bounds the $\phi$-divergence between the unknown probability distribution and the empirical distribution induced by cells. 

Many DRO problems are computationally intractable even though the ambiguity set is well defined and convex. 
Therefore, convex approximation schemes are proposed for certain types of DRO problems.
\citet{goh2010_DRO-tract-approx} studied two-stage distributionally robust linear programs with expectations in the objective 
and constraints. They developed an approximation framework based on linear decision rules that can reformulate the DRO problems  
into tractable conic programs if the ambiguity set is conic representable.  
\citet{wiesemann2015_dist-rob-conv-opt} proposed a framework for modeling and solving distributionally robust convex optimization problems 
in which the ambiguity set is conically representable and constraint functions are piecewise affine in both decision variables and random parameters. 
They showed that the reformulated problem is polynomial-time solvable under a strict nesting condition of the confidence sets. 
\citet{sim2017-DRO-infinity-constr-amb-sets} investigated DRO with the ambiguity set of probability distributions 
that can be characterized by a tractable conic representable support set with expectation constraints. 
This ambiguity set leads to a reformulation of DRO with a convex piecewise affine objective function 
as a tractable conic program \citep{wiesemann2015_dist-rob-conv-opt}. 
The conic constraint involved in this ambiguity set can be reformulated as 
infinitely many constraints induced by elements in the dual cone. Based on this reformulation technique, 
they proposed an iterative approach for this class of DRO problems by solving a sequence of tractable problems with finitely many constraints.
\citet{sim2017-adjust-RO-via-FM-elim} used the Fourier-Motzkin elimination technique in the two-stage adjustable robust optimization (ARO) 
setting to eliminate all or a subset of second stage variables sequentially, and remove some redundant constraints afterwards. 
In the cases where all second stage variables are eliminated, the two-stage ARO problems become classical robust linear programs, 
which can be solved to optimality. In the cases where a subset of second stage variables are eliminated,  
this technique can improve the solutions of two-stage ARO. 
\citet{bertsimas2017-solve-adapt-DRO-LP} developed a tractable framework for solving two-stage adaptive distributionally robust linear optimization problems with second-order conic representable ambiguity sets. It is shown that the two-stage adaptive distributionally robust linear optimization problem can be reformulated as a classical robust optimization problem, and a tractable formulation can be obtained by imposing linear decision rules (LDR) on the second stage variables. They also improved the current LDR techniques applied to adaptive distributionally robust linear optimization by incorporating uncertain parameters in the LDR setting.

\section{Reformulation of (D$^3$RO)}
\label{sec:d3ro_reform}
We investigate the dual of the inner problem of \eqref{opt:DRO} 
under the assumption that the probability distributions have a finite support on $\Xi$.
For some types of ambiguity sets considered in this section, the reformulation can be generalized to the 
case where $\Xi$ is a continuous support. The results for these more general cases are given in the appendix.
We make the following assumption:  
\begin{assumption}\label{ass:finite-support}
Every $P\in\mathcal{P}(\x)$ has a decision independent finite support $\Xi:=\{ \xivec^k \}^{N}_{k=1}$ in $\Xi$, $\forall \x\in X$, for a fixed $N$.
\end{assumption}
It follows that the candidate probability distributions in $\mathcal{P}(\x)$ can be represented as a vector $\bs{w}\in\mathbb{R}^N$ such that
$w_i$ is the mass assigned to the point $\xivec^i$ ($i\in[N]$) and $\|\bs{p} \|_1=1$ for all $\x\in X$. Note that the support of the distribution
and the number of scenarios are allowed to change with $x$ by forcing certain scenarios to have zero probability. 
In Sections~\ref{sec:reform-simple-bound}-\ref{sec:reform-KS},
we derive the dual of \eqref{opt:DRO}-inner to reformulate \eqref{opt:DRO} with the five different types of ambiguity sets discussed in Section~\ref{sec:d3ro_introd}.
Ambiguity sets are defined by simple measure and moment inequalities (Section~\ref{sec:reform-simple-bound}),
using bounds on moment constraints (Section~\ref{sec:reform-moment}), by Wasserstein metric (Section~\ref{sec:reform-wass}), using $\phi$-divergence (Section~\ref{sec:reform-phi-diverg}), and based on the K-S test (Section~\ref{sec:reform-KS}).

\subsection{Ambiguity sets defined by simple measure and moment inequalities}
\label{sec:reform-simple-bound}
%The ambiguity set can be defined using first, second, or higher moment inequalities \citep{bertsimas2010_models-minmax-stoch-LP-risk-aver,
%delage2010_distr-rob-opt-mom-uncer, mehrotra2015,  shapiro2004_class-of-minimax-analysis-stoch-progs, 
%shapiro2002_minimax-analysis-stoch-probs}.
We consider the moment robust set defined as follows:
\begin{equation}\label{def:moment-rob-set}
\mathcal{P}^{SM}(\x):=\Set*{P\in\M}{ \nu_1(\x)\preceq P \preceq \nu_2(\x),\; \int_{\Xi} f_i(\xi) P(d\xi) \in [l_i(\x), u_i(\x)] \quad  i\in[m] },
\end{equation}
where $\M$ is the set of positive measures defined on $(\Xi,\mathcal{F})$, 
$\nu_1(\x),\nu_2(\x)\in\M$ are two given measures for a fixed $\x$ that are lower and upper bounds of candidate probability measures,
and $\bs{f}:=[f_1(\xivec),\dotsc,f_m(\xivec)]$ is a vector of moment functions.
To ensure that $P$ is a probability distribution, we set $l_1(\x)=u_1(\x)=1$ and $f_1(\xivec)=1$ in the above definition of $\mathcal{P}^{SM}(\x)$.  
For any $\xivec\in\Xi$, let $\xivec:=[\xi_1,\ldots,\xi_d]$. When stand moments are used, the $i$th ($i\in[m]$) entry of $\bs{f}$ has the form: $f_i(\xivec):=(\xi_1)^{k_{i1}}\cdot (\xi_2)^{k_{i2}} \cdots (\xi_d)^{k_{id}}$, where $k_{ij}$ is a nonnegative integer indicating the power of $\xi_j$ for the $i$th moment function. The framework also allows the use of generalized moments by choosing alternative base functions.
Note that the first constraint in \eqref{def:moment-rob-set} is used to ensure that $P$ is a probability distribution. 
The ambiguity set \eqref{def:moment-rob-set} is a generalization of the set in \citep{mehrotra2015} for the decision dependent case. 
The following theorem gives a reformulation of \eqref{opt:DRO} with moment robust ambiguity set $\mathcal{P}^{SM}(\bs{x})$.
\begin{theorem}\label{thm:reform-simple-moment}
Let Assumption~\ref{ass:finite-support} hold. 
In the ambiguity set \eqref{def:moment-rob-set}, 
let $\nu_1(\x)=\sum^N_{i=1}\underline{p}_i(\x)\delta_{\xivec^i}$ and 
$\nu_2(\x)=\sum^N_{i=1}\overline{p}_i(\x)\delta_{\xivec^i}$ with $\underline{p}_i(\x)\le\overline{p}^i(\x)$ for $i\in[N]$.
If for any $\x\in X$ the ambiguity set \eqref{def:moment-rob-set} is nonempty,
then the \eqref{opt:DRO} problem with the ambiguity set $\mathcal{P}^{SM}(\x)$
can be reformulated as the following nonlinear program:
\begin{equation}\label{opt:reform-DRO-moment-constr}
\begin{aligned}
\underset{\bs{x},\bs{\alpha},\bs{\beta},\bs{\gamma},\bs{\mu}}{\emph{min}} & \quad f(\bs{x}) + \bs{\alpha}^T\bs{l}(\x) + \bs{\beta}^T\bs{u}(\x) 
		+\bs{\gamma}^T\underline{\bs{p}}(\x) + \bs{\mu}^T\overline{\bs{p}}(\x) \\
\emph{s.t.} &\quad  (\bs{\alpha}+\bs{\beta})^T \bs{f}(\bs{\xi}^k) + \gamma_k + \mu_k \ge h(\bs{x},\bs{\xi}^k) \quad \forall k\in [N], \\
 & \quad  \bs{x}\in X,\; \bs{\alpha},\bs{\beta}\in\mathbb{R}^{m},\;  
 \; \bs{\gamma}, \bs{\mu}\in\mathbb{R}^N,\; \bs{\alpha}\ge \bs{0},\; \bs{\beta}\le\bs{0},\; \bs{\gamma}\ge0, \; \bs{\mu}\le 0. 
\end{aligned}
\end{equation}
\end{theorem}
\begin{proof}
Under Assumption~\ref{ass:finite-support}, the inner problem of \eqref{opt:DRO} becomes the following linear program:
\begin{equation}\label{eqn:inner-moment-constr}
\begin{aligned}
&\underset{\bs{p}\in\mathbb{R}^{N}}{\textrm{max}}\;\; \sum^{N}_{k=1} p_k h(\bs{x},\bs{\xi}^k) \\
&\;\textrm{ s.t.}\quad \bs{l}(\bs{x})\le \sum^{N}_{k=1} p_k\bs{f}(\bs{\xi}^k) \le \bs{u}(\bs{x}), \\
&\qquad\quad \underline{p}_k(\x) \le p_k\le \overline{p}_k(\x)  \;\;\forall k\in[N].
\end{aligned}
\end{equation}
Based on the hypothesis of the theorem, the above linear program is feasible for any $\x\in X$.
We take the dual of \eqref{eqn:inner-moment-constr} and combine the dual problem with the outer problem to get the desired reformulation. 
\end{proof}
A reformulation for the two-stage case is given in the following corollary.
\begin{corollary}
If $h(\cdot,\cdot)$ is a recourse function defined in \eqref{eqn:h_recourse}, then the \eqref{opt:DRO} problem with 
the ambiguity set $\mathcal{P}^{SM}(\bs{x})$ can be formulated as follows:
\begin{equation}\label{opt:reform-DRO-moment-constr-recourse}
\begin{aligned}
\underset{\bs{x},\bs{y},\bs{\alpha},\bs{\beta},\bs{\gamma},\bs{\mu}}{\emph{min}} & \quad f(\bs{x}) + \bs{\alpha}^T\bs{l}(\x) + \bs{\beta}^T\bs{u}(\x) 
		+\bs{\gamma}^T\underline{\bs{p}}(\x) + \bs{\mu}^T\overline{\bs{p}}(\x)  \\
\emph{s.t.} &\quad  (\bs{\alpha}+\bs{\beta})^T \bs{f}(\bs{\xi}^k) + \gamma_k + \mu_k \ge g(\bs{x},\bs{y}^k,\bs{\xi}^k) \quad \forall k\in [N] \\
 & \quad \psi_i(\bs{x},\bs{y}^k,\bs{\xi}^k)\ge 0 \quad \forall i\in[m],\;\forall k\in[N], \\
 & \quad  \bs{x}\in X,\; \bs{\alpha},\bs{\beta}\in\mathbb{R}^{m},\;  
 \; \bs{\gamma}, \bs{\mu}\in\mathbb{R}^N,\; \bs{\alpha}\ge \bs{0},\; \bs{\beta}\le\bs{0},\; \bs{\gamma}\ge0, \; \bs{\mu}\le 0. 
\end{aligned}
\end{equation}
\end{corollary}

\subsection{Ambiguity sets defined by bounds on moment constraints}
\label{sec:reform-moment}
We now consider a moment robust set with multi-variate bounds defined as follows:
\begin{equation}\label{def:moment-bd-measure}
\mathcal{P}^{DY}(\x):=\Set*{P\in\Pset}{ \begin{aligned}  (\E_P[\xivec] - \bs{\mu}(\x))^T\bs{Q}(\x)^{-1}(\E_P[\xivec] - \bs{\mu}(\x)) \le \alpha(\x)\\
								\E_P[(\xivec-\bs{\mu}(\x))(\xi-\bs{\mu}(\x))^T]\preceq \beta(\x) \bs{Q}(\x)  \end{aligned}  }.
\end{equation} 
This set is a generalization of the set used in \citep{delage2010_distr-rob-opt-mom-uncer} for the decision dependent case.
Note that in a special case of \eqref{def:moment-bd-measure}, $\bs{\mu}(\x)$ and $\bs{Q}(\x)$ may not depend on $\x$. 
In this case, the confidence region specified by $\alpha(\x)$ and $\beta(\x)$ captures 
decision dependent ambiguity in estimating the distribution moments.
The following theorem gives a reformulation of \eqref{opt:DRO} with the ambiguity set $\mathcal{P}^{DY}(\bs{x})$.
This theorem is a generalization of Lemma~1 in \citep{delage2010_distr-rob-opt-mom-uncer} for the finite support case. 

\begin{theorem}
Let Assumption~\ref{ass:finite-support} hold.  
Suppose that Slater's constraint qualification conditions are satisfied, i.e., for any $\bs{x}\in X$, there exist a vector $\bs{p}^{\prime}:=[p^{\prime}_1,\ldots,p^{\prime}_N]$ such that $\sum^N_{i=1}p^{\prime}_i=1$,  
$\left(\sum^N_{i=1}p^{\prime}_i\bs{\xi}^i-\bs{\mu}(\bs{x})\right)^T\bs{Q}(\bs{x})^{-1}\left(\sum^N_{i=1}p^{\prime}_i\bs{\xi}^i-\bs{\mu}(\bs{x})\right)< \alpha(\bs{x})$ and $\sum^N_{i=1}p^{\prime}_i(\bs{\xi}^i-\bs{\mu}(\bs{x}))(\bs{\xi}^i-\bs{\mu}(\bs{x}))^T\prec\beta(\bs{x})\bs{Q}(\bs{x})$.
Then the \eqref{opt:DRO} problem with the ambiguity set $\mathcal{P}^{MB}(\x)$ can be reformulated as:
\begin{equation}\label{opt:reform-DRO-moment-bound-measure}
\begin{aligned}
\underset{\x,s,\bs{u},\bs{z},\bs{Y}}{\emph{min}}&\quad f(\bs{x}) +  s + [\sqrt{\alpha(\x)}, -(\bs{Q}(\x)^{-1/2}\bs{\mu}(\bs{x}))^T]\bs{z} + \beta(\x)\bs{Q}(\x)\bullet\bs{Y}  \\
\emph{s.t.}&\quad  s - (\xivec^{i})^T\bs{Q}(\x)^{-1/2}\bs{z}_1 + (\xivec^i-\bs{\mu}(\x))(\xivec^i-\bs{\mu}(\x))^T\bullet\bs{Y} \ge h(\x,\xivec^i) \quad \forall\; i\in[N], \\
&\quad \bs{x}\in X,\; \bs{z}:=[z_0,\bs{z}_1]\in\mathcal{N}_{\emph{SOC}},\;  \bs{Y}\in\mathbb{R}^{N\times N}, \; \bs{Y} \succeq 0,
\end{aligned}
\end{equation}
where $\mathcal{N}_{\emph{SOC}}$ is a second order cone defined as $\mathcal{N}_{\emph{SOC}}:=\Set*{\bs{y}:=[y_0,y_1,\ldots,y_d]}{y_0\ge \sqrt{\sum^d_{i=1}y^2_i}}$,
$\bs{z}:=[z_0,\;\bs{z}_1]$ with $z_0\in\mathbb{R}$ and $\bs{z}_1\in\mathbb{R}^d$, and $A\bullet B=Tr(A^TB)$ for matrices $A$ and $B$. 
\end{theorem}
\begin{proof}
Under Assumption~\ref{ass:finite-support}, the ambiguity set becomes:
\begin{equation}
\mathcal{P}^{MB}(\x):=\Set*{\bs{p}\in\mathbb{R}^N}{ \begin{aligned} & \sum^{N}_{i=1} p_i=1,\quad \bs{\tau}=\sum^{N}_{i=1} p_i\xivec^i, \\ & (\bs{\tau}-\bs{\mu}(\x))^T\bs{Q}(\x)^{-1}(\bs{\tau}-\bs{\mu}(\x))\le \alpha(\x),  \\ & \sum^{N}_{i=1} p_i(\xivec^i - \bs{\mu}(\x))(\xivec^i-\bs{\mu}(\x))\preceq \beta(\x)\bs{Q}(\x)  \end{aligned}   }.
\end{equation}
Then the inner problem of \eqref{opt:DRO} can be formulated as:
\begin{equation}\label{opt:inner-moment-measure}
\begin{aligned}
\underset{\bs{p},\tau}{\textrm{max}} &\quad \sum^{N}_{i=1} p_i h(\x,\xivec^i) & \\
\textrm{s.t.} & \quad \sum^{N}_{i=1} p_i = 1,  & :\; s \in \mathbb{R}, \\
& \quad \bs{\tau} = \sum^{N}_{i=1} p_i \xivec^i,   & :\; \bs{u}\in\mathbb{R}^d, \\
& \quad (\bs{\tau}-\bs{\mu}(\x))^T\bs{Q}(\x)^{-1}(\bs{\tau}-\bs{\mu}(\x)) \le \alpha(\x),  & :\; \bs{z}\in\mathcal{K}_{\textrm{SOC}}, \\
& \quad \sum^{N}_{i=1} p_i(\xivec^i-\bs{\mu}(\x))(\xivec^i-\bs{\mu}(\x))^T \preceq \beta(\x)\bs{Q}(\x),  & :\; \bs{Y}\succeq 0, \\  
& \quad \bs{p}:=[p_1,\ldots,p_{N}]\in\mathbb{R}^{N}, & \forall i\in[N].
\end{aligned}
\end{equation}
The Lagrangian function of \eqref{opt:inner-moment-measure} using the dual variables indicated in \eqref{opt:inner-moment-measure} has the following form:
\begin{equation}\label{eqn:lagrangian}
\begin{aligned}
& L(\bs{p},\bs{\tau};s,\bs{u},\bs{z},\bs{Y}) = \sum^{N}_{i=1} p_i h(\bs{x},\xivec^i) - s \left(\sum^{N}_{i=1} p_i - 1\right)
  + \bs{u}^T\left( \bs{\tau} - \sum^{N}_{i=1} p_i\xivec^i \right)  \\ 
&\hspace{0.5cm} + [\sqrt{\alpha(\x)}, (\bs{Q}(\x)^{-1/2}(\bs{\tau} - \bs{\mu}(\x)))^T] \bs{z} 
  + \left[  \beta(\x)\bs{Q}(\x) - \sum^{N}_{i=1}p_i(\xivec^i-\bs{\mu}(\x))(\xivec^i-\bs{\mu}(\x))^T \right]\bullet \bs{Y}   \\
&= \sum^{N}_{i=1} \left[ h(x,\xivec^i) - s - \bs{u}^T\xivec^i \right]p_i + \left( \bs{u} + \bs{Q}(\x)^{-1/2}\bs{z}_1 \right)^T\bs{\tau}+ s + \beta(\x)\bs{Q}(\x)\bullet\bs{Y} 
\\
&\quad  + \left[\sqrt{\alpha(\x)},  -\left(\bs{Q}(\x)^{-1/2}\bs{\mu}(\x)\right)^T \right]\bs{z}.
\end{aligned}
\end{equation}
Applying Sion's minimax theorem, the inner problem \eqref{opt:inner-moment-measure} has the form:
\begin{equation}\label{opt:lag-dual-inner-moment-measure}
\underset{s,\bs{u},\bs{z},\bs{Y}}{\textrm{min}}\; \left\{ \underset{\bs{p},\bs{\tau}}{\textrm{max}}\; \left\{ L(\bs{p},\bs{\tau}; s,\bs{u},\bs{z},\bs{Y}):\; \bs{p}\ge 0,\; \bs{\tau}\in\mathbb{R}^d \right\} \right\}.
\end{equation} 
Since Slater's constraint qualification conditions are satisfied, strong duality holds, 
and hence \eqref{opt:lag-dual-inner-moment-measure} and \eqref{opt:inner-moment-measure} have the same optimal value. 
Substituting the Lagrangian function \eqref{eqn:lagrangian} into \eqref{opt:lag-dual-inner-moment-measure}, 
and solving the inner maximization problem over $\bs{p}$ and $\tau$, we reformulate 
the dual problem as:
\begin{equation}\label{opt:dual-inner-moment-measure}
\begin{aligned}
\underset{s,\bs{u},\bs{z},\bs{Y}}{\textrm{min}} &\quad s + \beta(\x)\bs{Q}(\x)\bullet\bs{Y} + \big[\sqrt{\alpha(\x)},  -\left(\bs{Q}(\x)^{-1/2}\bs{\mu}(\x)\right)^T \big]\bs{z} \\
\textrm{s.t.} &\quad  s + \bs{u}^T\xivec^i \ge h(\x,\xivec^i) \qquad \forall i\in[N], \\
& \quad \bs{u} + \bs{Q}(\x)^{-1/2}\bs{z}_1 = 0, \\
& \quad \bs{z}\in\mathcal{K}_{\textrm{SOC}},\; \bs{Y}\in\mathbb{R}^{N\times N}, \; \bs{Y}\succeq 0.          
\end{aligned}
\end{equation}
Substituting \eqref{opt:dual-inner-moment-measure} in \eqref{opt:DRO}, we obtain \eqref{opt:reform-DRO-moment-bound-measure}.   
\end{proof}

\begin{corollary}
If $h(\cdot,\cdot)$ is a recourse function defined in \eqref{eqn:h_recourse}, then the \eqref{opt:DRO} problem 
with the ambiguity set $\mathcal{P}^{MB}(\x)$ can be reformulated as follows:
\begin{equation}\label{opt:reform-DRO-moment-bound-measure-recourse}
\begin{aligned}
\underset{\x,s, \bs{y}, \bs{u},\bs{z},\bs{Y}}{\emph{min}}&\quad f(\bs{x}) +  s + [\sqrt{\alpha(\x)}, -(\bs{Q}(\x)^{-1/2}\bs{\mu}(\bs{x}))^T]\bs{z} + \beta(\x)\bs{Q}(\x)\bullet\bs{Y}  \\
\emph{s.t.}&\quad  s - (\xivec^{k})^T\bs{Q}(\x)^{-1/2}\bs{z}_1 + (\xivec^k-\bs{\mu}(\x))(\xivec^k-\bs{\mu}(\x))^T\bullet\bs{Y} \ge g(\x,\bs{y}^k,\xivec^k) \quad \forall\; k\in[N], \\
& \quad \psi_i(\bs{x},\bs{y}^k,\bs{\xi}^k)\ge 0 \quad \forall i\in[m],\;\forall k\in[N], \\
&\quad \bs{x}\in X,\; \bs{z}:=[z_0,\bs{z}_1]\in\mathcal{N}_{\emph{SOC}},\; \bs{Y} \succeq 0.
\end{aligned}
\end{equation}
\end{corollary}

\subsection{Ambiguity sets defined by Wasserstein metric}
\label{sec:reform-wass}
Instead of using moment based definitions of the ambiguity set, we may 
define this set using a statistical distance, such as the Wasserstein metric. 
We now study the \eqref{opt:DRO} problem with a decision dependent ambiguity set defined using the $L_1$-Wasserstein metric as follows:
\begin{equation}\label{def:wass-metric}
\mathcal{P}^{W}(\x):=\Set*{P\in\Pset}{\W(P,P_0)\le r(\x)}, 
\end{equation}
where $P_0$ is a nominal probability distribution, and $\W(\cdot,\cdot):\;\Pset\times\Pset\to\mathbb{R}$ 
is the $L_1$-Wasserstein metric defined in \citep{givens1984}:
\begin{equation} \label{def:Kantorovich_metric}
\W(P_1, P_2):= \underset{K\in\mathcal{S}(P_1,P_2)}{\textrm{inf}} \int_{\Xi\times \Xi} \| \bs{s}_1 - \bs{s}_2 \| K(d\bs{s}_1\times d \bs{s}_2), 
\end{equation}
where $\mathcal{S}(P_1,P_2):=\big\{ K\in\mathcal{P}(\Xi\times\Xi,\mathcal{F}\times\mathcal{F}):\; K(A\times\Xi) = P_1(A),\;  K(\Xi\times A) = P_2(A),\; \forall A\in\mathcal{F}   \big\}$ is the set of all joint probability distributions whose marginals are $P_1$ and $P_2$, 
and $\|\cdot\|$ is an arbitrary norm defined on $\mathbb{R}^d$.
The ambiguity set \eqref{def:wass-metric} is a generalization of the one considered in 
\citep{gao2016-dro-wass-distance,esfahani2015_distr-rob-wass-metric,s-abadeh2015_distr-rob-log-reg,luo2017-decomp-alg-distr-rob-opt}
for the decision dependent case.
As a special case of \eqref{def:wass-metric}, under Assumption~\ref{ass:finite-support}, $\mathcal{P}^W(\x)$ is written as:
\begin{equation}\label{def:wass-metric-finite}
P^W(\x)=\{\bs{p}\in\mathbb{R}^N\;|\; \mathcal{W}(\bs{p},\hat{\bs{p}})\le r(\x),\; \sum^N_{i=1}p_i=1,\; p_i\ge 0,\; \forall i\in[N] \},
\end{equation} 
where $\hat{\bs{p}}$ is a given empirical probability distribution on $\Xi$, and the Wasserstein metric can be simplified as $\mathcal{W}(\bs{p},\hat{\bs{p}})=\big\{\underset{\bs{w}}{\textrm{min}}\; \sum^N_{i=1}\sum^N_{j=1}\|\bs{\xi}^i-\bs{\xi}^j\|w_{ij}\;\; \big\vert \;
 \sum^N_{j=1}w_{ij}=p_i \; \forall i\in[N], \; \sum^N_{i=1}w_{ij}=\hat{p}_i\; \forall j\in[N],\; w_{ij}\ge 0\;\forall i,j\in[N] \big\}$.
The following theorem gives a reformulation of \eqref{opt:DRO} for the ambiguity set \eqref{def:wass-metric-finite}.
\begin{theorem}
Let Assumption~\ref{ass:finite-support} hold. In the ambiguity set \eqref{def:wass-metric}, let
the reference distribution $P_0$ be: $P_0=\sum^N_{i=1}\hat{p}_i\delta_{\bs{\xi}^i}$, 
then the \eqref{opt:DRO} problem with the ambiguity set \eqref{def:wass-metric-finite}
can be reformulated as:
\begin{equation}\label{opt:reform-DRO-wass}
\begin{aligned}
&\underset{\bs{x},\bs{\alpha},\bs{\beta},\bs{\mu},\bs{\lambda},\gamma,\eta}{\emph{min}}\;\; f(\bs{x})-\sum^N_{i=1}\hat{p}_i\beta_i - r(\bs{x})\gamma + \eta \\
&\quad\emph{ s.t. }\; \alpha_i+\mu_i + \eta \ge h(\bs{x},\bs{\xi}^i) \qquad \forall i\in[N], \\
&\qquad\quad -\alpha_i+\beta_j + \|\bs{\xi}^i - \bs{\xi}^j\|\gamma + \lambda_{ij} \ge 0 \qquad \forall i\in[N],\;\forall j\in[N],  \\
&\qquad\quad \bs{x}\in X, \;\; \alpha_i\in\mathbb{R},\;\beta_i\in\mathbb{R},\;\mu_i\ge 0,\;\lambda_{ij}\ge 0,\;\gamma\le 0,\;\eta\in\mathbb{R}\quad \forall i\in[N],\;\forall j\in[N].
\end{aligned}
\end{equation}
\end{theorem}

\begin{proof}
Since $\Xi$ is finite, the \eqref{opt:DRO}-inner problem with ambiguity set $\mathcal{P}^W(\bs{x})$ 
can be formulated as the following linear program:
\begin{equation}\label{opt:inner-reform-wass}
\begin{aligned}
&\underset{\bs{p},\bs{w}}{\textrm{max}}\;\;\sum^N_{i=1}h(\bs{x},\bs{\xi}^i)p_i    \\ 
&\;\textrm{ s.t. } \; \sum^N_{j=1}w_{ij}=p_i \qquad \forall i\in [N],  \\
&\qquad \sum^N_{i=1}w_{ij}=\hat{p}_j \qquad \forall j\in[N], \\
&\qquad \sum^N_{i=1}\sum^N_{j=1} \|\bs{\xi}^i - \bs{\xi}^j\|w_{ij} \le r(\bs{x}),  \\
&\qquad \sum^N_{i=1}p_i=1,\;\; p_i\ge 0,\;w_{ij}\ge 0\;\; \forall i\in[N],\;\forall j\in[N],   
\end{aligned}
\end{equation}   
where $\bs{w}$ is a joint probability distribution with two marginal distributions given by $\bs{p}$ and $\hat{\bs{p}}$, respectively.
The dual of the above linear program is:
\begin{equation}\label{opt:inner-dual-wass}
\begin{aligned}
&\underset{\bs{\alpha},\bs{\beta},\bs{\mu},\bs{\lambda},\gamma,\eta}{\textrm{min}}\;\; -\sum^N_{i=1}\hat{\bs{p}}_i\beta_i - r(\bs{x})\gamma + \eta \\
&\quad\textrm{ s.t. }\; \alpha_i+\mu_i + \eta \ge h(\bs{x},\bs{\xi}^i) \qquad \forall i\in[N], \\
&\qquad\quad -\alpha_i+\beta_j + \|\bs{\xi}^i - \bs{\xi}^j\|\gamma + \lambda_{ij} \ge 0 \qquad \forall i\in[N],\;\forall j\in[N],  \\
&\qquad\quad \alpha_i\in\mathbb{R},\;\beta_i\in\mathbb{R},\;\mu_i\ge 0,\;\lambda_{ij}\ge 0,\;\gamma\le 0,\;\eta\in\mathbb{R}\quad \forall i\in[N],\;\forall j\in[N].
\end{aligned}
\end{equation}
After substituting \eqref{opt:inner-dual-wass} into \eqref{opt:DRO}, we obtain the desired reformulation \eqref{opt:reform-DRO-wass}.
\end{proof}
A reformulation for the two-stage case is given in the following corollary.
\begin{corollary}
If $h(\cdot,\cdot)$ is a recourse function defined in \eqref{eqn:h_recourse}, then the \eqref{opt:DRO} problem 
with the ambiguity set $\mathcal{P}^{W}(\x)$ can be reformulated as follows:
\begin{equation}\label{opt:reform-DRO-wass-recourse}
\begin{aligned}
&\underset{\bs{x},\bs{y},\bs{\alpha},\bs{\beta},\bs{\mu},\bs{\lambda},\gamma}{\emph{min}}\;\; f(\bs{x})-\sum^N_{k=1}\hat{\bs{p}}_k\beta_k - r(\bs{x})\gamma \\
&\quad\emph{ s.t. }\; \alpha_k+\mu_k\ge g(\bs{x},\bs{y}^k,\bs{\xi}^k) \qquad \forall i\in[N], \\
&\qquad\quad -\alpha_i+\beta_j + d(\bs{\xi}^i,\bs{\xi}^j)\gamma + \lambda_{ij} \ge 0 \qquad \forall i\in[N],\;\forall j\in[N],  \\
&\qquad\qquad \psi_i(\bs{x},\bs{y}^k,\bs{\xi}^k)\ge 0 \quad \forall i\in[m],\;\forall k\in[N], \\
&\qquad\quad \bs{x}\in X,\;\; \alpha_i\in\mathbb{R},\;\beta_i\in\mathbb{R},\;\mu_i\ge 0,\;\lambda_{ij}\ge 0,\;\gamma\le 0\quad \forall i\in[N],\;\forall j\in[N].
\end{aligned}
\end{equation}
\end{corollary}
The reformulation \eqref{opt:reform-DRO-wass} of \eqref{opt:DRO} with the ambiguity set defined using the Wasserstein metric 
can be generalized for the case where the support $\Xi$ is continuous.
The details of this generalization are given in Appendix~\ref{appendix-Wass}.

\subsection{Ambiguity sets defined using $\phi$-divergence}
\label{sec:reform-phi-diverg}
We now study the \eqref{opt:DRO} problem using a decision dependent ambiguity set defined using the notion of $\phi$-divergence:
\begin{equation}\label{def:phi-diverg}
\mathcal{P}^{\phi}(\bs{x}):= \big\{ P\in\mathcal{P}(\Xi,\mathcal{F}):\; \mathcal{D}_{\phi}(P||P_0)\le \eta(\bs{x}) \big\},
\end{equation}
where $\mathcal{D}_{\phi}(P||P_0)=\int_{\Omega}\phi\left(\frac{dP}{dP_0}\right)dP_0$, and $\phi$ is a non-negative and convex function. 
This type of ambiguity set is a generalization of the one considered in    
\citep{ben-tal2013_rob-opt-uncert-prob,calafiore2007_amb-risk-meas-opt-rob-portf,
jiang2015_risk-aver-two-stage-stoch-prog-distr-amb,love2016_phi-diverg-constr-amb-stoch-prog-data-driven-opt,
wang2016_likehd-rob-opt-data-driven,yanikoglu2013_safe-approx-amb-constr-hist-data} for the decision dependent case. 
Under Assumption~\ref{ass:finite-support}, and using $P_0:=\sum^N_{i=1}\hat{p}_i\delta_{\bs{\xi}_i}$ as the nominal distribution,
the ambiguity set \eqref{def:phi-diverg} is written as:
\begin{equation}\label{def:phi-diverg-discrete}
\mathcal{P}^{\phi}(\bs{x})= \left\{ P=\sum^N_{i=1}p_i\delta_{\bs{\xi}^i}:\; \sum^N_{i=1}\hat{p}_i\phi(p_i/\hat{p}_i) \le \eta(\bs{x}), \; 
\sum^N_{i=1}p_i=1, \; p_i\ge 0 \;\;\forall i\le N \right\}.
\end{equation}
Two reformulations of \eqref{opt:DRO} with ambiguity set $\mathcal{P}^{\phi}(\bs{x})$ are given in the following theorem.
\begin{theorem}
Let Assumption~\ref{ass:finite-support} hold, and $\phi$ be a non-negative convex function. 
Assume that the following Slater condition is satisfied for every $\bs{x}\in X$:
there exist a $\bs{p}\in\mathbb{R}^N$ such that $p_i>0$, $\sum^N_{i=1}p_i=1$ and $\sum^N_{i=1}\hat{p}_i\phi(p_i/\hat{p}_i)<\eta(\bs{x})$.
Then \eqref{opt:DRO} with the ambiguity set $\mathcal{P}^{\phi}(\bs{x})$ can be reformulated as the following semi-infinite program:
\begin{equation}\label{opt:reform-phi-diverg-Lag}
\begin{aligned}
&\underset{\x,\bs{p},\alpha,\beta,\bs{\lambda},z}{\emph{min}} \;\; z \\
&\;\;\emph{ s.t. }\;\; z\ge \sum^N_{i=1}h(\bs{x},\bs{\xi}^i)p_i + 
\alpha\Big( \frac{1}{N}\sum^N_{i=1}\phi(Np_i)-\eta(\bs{x}) \Big) + \beta\Big(\sum^N_{i=1}p_i-1\Big) + \sum^N_{i=1}p_i\lambda_i \qquad \forall \bs{p}\in S,\\
&\qquad\quad  \bs{x}\in X,\;\alpha\le 0,\; \beta\in\mathbb{R},\; \lambda_i\le 0,\; p_i\ge0\;\;\forall i\in[N],
\end{aligned}
\end{equation}
where $S=\{\bs{p}\in\mathbb{R}^N:\; \sum^N_{i=1}p_i=1,\;p_i\ge 0\;\;\forall i\in[N]\}$.
Alternatively, \eqref{opt:reform-phi-diverg-Lag} also has the reformulation:
\begin{equation}\label{opt:reform-phi-diverg}
\begin{aligned}
&\underset{\bs{x},\bs{p},\alpha,\beta,\bs{\lambda}}{\emph{min}}\;\; f(\bs{x}) + \sum^N_{i=1}p_ih\big(\bs{x},\bs{\xi}^i\big) 
     + \alpha\Big(\frac{1}{N}\sum^N_{i=1}\phi(Np_i)-\eta(\bs{x}) \Big) + \beta\Big(\sum^N_{i=1}p_i-1\Big) + \sum^N_{i=1}\lambda_ip_i \\
&\;\;\emph{ s.t. }\;\ \alpha \phi^{\prime}(Np_i) +\beta + h(\bs{x},\bs{\xi}^i) + \lambda_i = 0 \qquad \forall i\in[N], \\
&\qquad\quad \bs{x}\in X,\;\alpha\le 0,\; \beta\in\mathbb{R},\; \lambda_i\le 0,\; p_i\ge0\;\;\forall i\in[N].
\end{aligned}
\end{equation}
\end{theorem}
\begin{proof}
The \eqref{opt:DRO} problem can be written as $\underset{\bs{x}\in X}{\textrm{min}}\;\; f(\bs{x}) + \Phi(\bs{x})$,
where the function $\Phi(\bs{x})$ is the optimal objective of the following optimization problem:
\begin{equation}\label{opt:reform-phi-diverg-1}
\begin{aligned}
\Phi(\bs{x})=\;&\underset{\bs{p}}{\textrm{max}}\;\sum^N_{i=1}p_ih\big(\bs{x},\bs{\xi}^i\big) \\
&\textrm{ s.t. }\;\; \frac{1}{N}\sum^N_{i=1}\phi(Np_i) \le \eta(\bs{x}), \quad  \sum^N_{i=1}p_i=1, \\
& \qquad p_i\ge 0 \;\;\forall i\le N, \;\; \bs{x}\in X.
\end{aligned}
\end{equation}
Since $\phi$ is convex, \eqref{opt:reform-phi-diverg-1} is a convex program with respect to the decision variable $\bs{p}$.
For a fixed $\bs{x}\in X$, the Lagrangian dual of \eqref{opt:reform-phi-diverg-1} is written as follows:
\begin{equation}\label{opt:Lag-dual}
\begin{aligned}
&\underset{\alpha,\beta,\bs{\lambda}}{\textrm{min}}\; \underset{\bs{p}}{\textrm{max}}\;\mathcal{L}(\bs{p};\alpha,\beta,\bs{\lambda}) \\
&\textrm{ s.t. }  \alpha\le 0,\; \beta\in\mathbb{R}, \; \lambda_i\le 0,\quad \forall i\in[N],
\end{aligned}
\end{equation}
where $\mathcal{L}(\bs{p};\alpha,\beta,\bs{\lambda})=\sum^N_{i=1}h(\bs{x},\bs{\xi}^i)p_i + 
\alpha\Big( \frac{1}{N}\sum^N_{i=1}\phi(Np_i)-\eta(\bs{x}) \Big) + \beta\Big(\sum^N_{i=1}p_i-1\Big) + \sum^N_{i=1}p_i\lambda_i$,
and $\alpha$, $\beta$, $\bs{\lambda}$ are the Lagrangian multipliers.
Since Slater's condition is satisfied for any $\bs{x}\in X$, strong duality holds. 
The inner maximization problem of \eqref{opt:Lag-dual} is equivalent to $\{\textrm{max}\;z,\;\; \textrm{s.t. }z\ge \mathcal{L}(\bs{p};\alpha,\beta,\bs{\lambda})\;\;\forall \bs{p}\in S\}$, which gives the reformulation \eqref{opt:reform-phi-diverg-Lag}.
 
Note that the inner problem of \eqref{opt:Lag-dual} is an unconstrained convex optimization problem. 
Using the the KKT optimality conditions we have:
\begin{equation}
\frac{\partial\mathcal{L}}{\partial p_i}=\alpha\phi^{\prime}(Np_i)+h(\bs{x},\xi^i)+\beta+\lambda_i=0,\qquad \forall i\in[N].
\end{equation}
After substituting the expression of the Lagrangian in \eqref{opt:Lag-dual}, adding the optimality condition
and using strong duality, we obtain the reformulation given in \eqref{opt:reform-phi-diverg}.
\end{proof}

A reformulation for the two-stage stochastic optimization case is given in the following corollary.
\begin{corollary}
If $h(\cdot,\cdot)$ is a recourse function defined in \eqref{eqn:h_recourse}, then the \eqref{opt:DRO} problem 
with the ambiguity set $\mathcal{P}^{\phi}(\x)$ can be reformulated as follows:
\begin{equation}\label{opt:reform-phi-diverg-recourse}
\begin{aligned}
&\underset{\substack{\bs{x},\bs{y},z,\bs{p},\\ \alpha,\beta,\bs{\lambda}}}{\emph{min}}\;\; z  \\
&\;\;\emph{ s.t.}\;\ z\ge f(\bs{x}) + \sum^N_{k=1}p_kg\big(\bs{x},\bs{y}^k,\xi^k\big) 
     + \alpha\Big(\frac{1}{N}\sum^N_{k=1}\phi(Np_k)-\eta(\bs{x}) \Big) + \beta\Big(\sum^N_{k=1}p_k-1\Big) + \sum^N_{k=1}\lambda_kp_k \; \forall \bs{p}\in S, \\
&\qquad\quad \psi_i(\bs{x},\bs{y}^k,\bs{\xi}^k)\ge 0 \quad \forall i\in[m],\;\forall k\in[N], \\
&\qquad\quad z\in\mathbb{R},\;\bs{x}\in X,\;\alpha\le 0,\; \beta\in\mathbb{R},\; \lambda_k\le 0,\; p_k\ge0\;\;\forall k\in[N],
\end{aligned}
\end{equation}
where $S=\{\bs{p}\in\mathbb{R}^N:\; \sum^N_{k=1}p_k=1,\;p_k\ge 0\;\;\forall k\in[N]\}$.
\end{corollary}

\subsection{Ambiguity sets defined based on the Kolmogorov-Smirnov test}
\label{sec:reform-KS}
The K-S distance has been used by \citet{bertsimas2013_data-driven-rob-opt} 
in defining an ambiguity set in data-driven robust optimization models. 
For two univariate probability distributions $P_1$ and $P_2$, let $F_1$ and $F_2$ be their cumulative
distribution functions. The Kolmogorov-Smirnov (KS) distance is defined as:
\begin{equation}\label{def:KS-univariate}
D(P_1,P_2)=\underset{s}{\textrm{sup}}\; |F_1(s)-F_2(s)|.
\end{equation} 
We now study the \eqref{opt:DRO} problem with the ambiguity set defined based on the KS-distance.
Note that although \eqref{def:KS-univariate} is defined for an univariate random variable, this definition can be directly 
generalized for the probability distribution of a random vector with a finite support.
Specifically, under Assumption~\ref{ass:finite-support}, let $P_0=\sum^N_{i=1}\hat{p}_i\delta_{\bs{\xi}_i}$
be an empirical probability distribution. 
The KS-distance between a discrete probability distribution $P=\sum^N_{i=1}p_i\delta_{\bs{\xi}_i}$
and $P_0$ can be written as: 
\begin{equation}
D(P,P_0)=\underset{k\in[N]}{\textrm{sup}}\; \Big| \sum^k_{i=1}p_i - \sum^k_{i=1}\hat{p}_i \Big|.
\end{equation}
The decision dependent ambiguity set of probability distributions is constructed using the KS-distance as follows:
\begin{equation}\label{eqn:KS-dist-discrete}
\mathcal{P}^{KS}(\bs{x})=\left\{\bs{p}\in\mathbb{R}^{N}:\;\;  
\underset{k\in[N]}{\textrm{sup}}\;\Big|\sum^k_{i=1}p_i-\sum^k_{i=1}\hat{p}_i \Big|\le \eta(\bs{x}),
\;\; \sum^N_{i=1}p_i=1,\;\; p_i\ge 0\;\;\forall i\in[N] \right\}. 
\end{equation}
A reformulation of the \eqref{opt:DRO} problem is given in the following theorem.
\begin{theorem}
Let Assumption~\ref{ass:finite-support} hold. The \eqref{opt:DRO} problem with the ambiguity set \eqref{eqn:KS-dist-discrete}
can be reformulated as:
\begin{equation}\label{opt:DRO-KS-dual}
\begin{aligned}
&\underset{\bs{x},\lambda,\bs{\alpha},\bs{\beta},\bs{\gamma}}{\emph{min}}\; f(\bs{x})+\lambda + \sum^N_{i=1}\sum^N_{k=1}(\alpha_k+\beta_k)\hat{p}_i
   +\sum^N_{k=1}(\alpha_k-\beta_k)\eta(\bs{x})  \\
&\quad\emph{ s.t. } \lambda + \sum^N_{k=i}(\alpha_k+\beta_k) + \gamma_i \ge h(\bs{x},\xi^i) \qquad \forall i\in[N], \\
&\qquad\quad \lambda\in\mathbb{R},\; \alpha_i\le0,\; \beta_i\ge0,\; \gamma_i\ge0\;\;\forall i\in[N].   
\end{aligned}
\end{equation}
\end{theorem}
\begin{proof}
The \eqref{opt:DRO} problem with the ambiguity set \eqref{eqn:KS-dist-discrete} can be written as:
\begin{equation}\label{eqn:DRO-KS-primal}
\begin{aligned}
&\underset{\bs{x}}{\textrm{min}}\;f(\bs{x}) + \underset{\bs{p}}{\textrm{max}}\;\sum^N_{i=1}h(\bs{x},\xi^i)p_i \\
&\;\textrm{ s.t. }\; \underset{k\in[N]}{\textrm{sup}}\;\Big|\sum^k_{i=1}p_i-\sum^k_{i=1}\hat{p}_i \Big|\le \eta(\bs{x}), \\
&\qquad \sum^N_{i=1}p_i=1, \quad p_i\ge 0 \quad \forall i\in[N].
\end{aligned}
\end{equation}
Note that the inner problem of \eqref{eqn:DRO-KS-primal} can be reformulated as the following linear program:
\begin{equation}
\begin{aligned}
&\underset{\bs{p}}{\textrm{max}}\; \sum^N_{i=1}h(\bs{x},\xi^i)p_i \\
&\textrm{ s.t. } \sum^k_{i=1}p_i-\sum^k_{i=1}\hat{p}_i\le\eta(\bs{x}) \qquad \forall k\in[N], \\
&\qquad \sum^k_{i=1}p_i-\sum^k_{i=1}\hat{p}_i\ge -\eta(\bs{x}) \qquad \forall k\in[N], \\ 
&\qquad \sum^N_{i=1}p_i=1,\quad p_i\ge 0 \quad \forall i\in[N].
\end{aligned}
\end{equation}
After taking the dual of the above linear program and combining it with the outer problem, we obtain \eqref{opt:DRO-KS-dual}.
\end{proof}

A reformulation for the two-stage stochastic optimization case is given in the following corollary.
\begin{corollary}
If $h(\cdot,\cdot)$ is a recourse function defined in \eqref{eqn:h_recourse}, then the \eqref{opt:DRO} problem 
with the ambiguity set $\mathcal{P}^{KS}(\x)$ can be reformulated as follows:
\begin{equation}\label{opt:DRO-KS-dual-recourse}
\begin{aligned}
&\underset{\bs{x},\bs{y},\lambda,\bs{\alpha},\bs{\beta},\bs{\gamma}}{\emph{min}}\; f(\bs{x})+\lambda + \sum^N_{i=1}\sum^N_{k=1}(\alpha_k+\beta_k)\hat{p}_i
   +\sum^N_{k=1}(\alpha_k-\beta_k)\eta(\bs{x})  \\
&\quad\emph{ s.t. } \lambda + \sum^N_{k=i}(\alpha_k+\beta_k) + \gamma_i \ge g(\bs{x},\bs{y}^i,\xi^i) \qquad \forall i\in[N], \\
&\qquad\quad \psi_i(\bs{x},\bs{y}^k,\bs{\xi}^k)\ge 0 \quad \forall i\in[m],\;\forall k\in[N], \\
&\qquad\quad \x\in X,\; \lambda\in\mathbb{R},\; \alpha_i\le0,\; \beta_i\ge0,\; \gamma_i\ge0\;\;\forall i\in[N].   
\end{aligned}
\end{equation}
\end{corollary}
The reformulation \eqref{opt:DRO-KS-dual} of \eqref{opt:DRO} with the ambiguity set defined using the K-S distance
can be generalized for the case where the support $\Xi$ is continuous.
The details of this generalization are given in Appendix~\ref{appendix-KS}.

\section{Concluding Remarks}
We have established a framework for reformulating the distributionally robust optimization problems 
with important types of decision dependent ambiguity sets.
These ambiguity sets contain decision dependent parameters.
For example, the moment robust ambiguity set \eqref{def:moment-bd-measure} contains parameters 
$\alpha(\x)$, $\beta(\x)$, $\bs{\mu}(\x)$ and $\bs{Q}(\x)$, which are functions of the decision $\x$. 
We now briefly discuss the estimation of these functions using a data-driven approach. 
Ambiguity sets for $\bs{\xi}$ under an arbitrary decision $\x$ can be constructed if such information is 
available from past decisions, or if it is possible for us to experiment with trial decisions $\{\x^i\}^k_{i=1}$ and collect samples of the random vector    
$\bs{\xi}$ under each decision $\x^i$. From these samples 
we can establish the analytical relation between the parameters in defining the ambiguity set and the decision using
statistical learning models. We can subsequently extrapolate this analytical relation to a general decision $\x$ to obtain an empirical
decision dependent ambiguity set description. 

The goal of this paper was to show that it is possible to extend the dual formulations in 
DRO even when the ambiguity sets are decision dependent. 
The analysis suggests that the situations for which DRO models admit a dual reformulation also allow for 
dual reformulations for the decision dependent case. 
The reformulated models are generally non-convex optimization problems 
requiring further investigation towards developing efficient algorithms for the specific situations.
The non-convex optimization problems may have further structure when additional assumptions on 
decision dependent parameters and the feasible set $X$ are imposed. 
This structure may be exploited for further refined reformulations and the development of efficient algorithms.

% Appendix here
% Options are (1) APPENDIX (with or without general title) or
%             (2) APPENDICES (if it has more than one unrelated sections)
% Outcomment the appropriate case if necessary
%
% \begin{APPENDIX}{<Title of the Appendix>}
% \end{APPENDIX}
%
%   or
%
\begin{appendices}
\section{Reformulation of \eqref{opt:DRO} with Wasserstein metric and continuous support of random parameters}
\label{appendix-Wass}
We study the reformulation of \eqref{opt:DRO} with Wasserstein metric and continuous support of random parameters.
In contrast to the case studied in Section~\ref{sec:reform-wass}, we do not assume that Assumption~\ref{ass:finite-support}
holds. As a consequence, the support $\Xi$ of the decision dependent random parameters $\xi$ can be continuous.
Suppose at a decision $\bs{x}_0$, we have observed $N$ samples of the random variable $\bs{\xi}$, written as
$\{\bs{\xi}^i\}^N_{i=1}$. We construct an empirical distribution as: $P_0=\sum^N_{i=1}\frac{1}{N}\delta_{\bs{\xi}^i}$.
Setting the empirical distribution as the center of the Wasserstein ball, we can define the decision dependent ambiguity set as:
\begin{equation}\label{def:wass-metric-cont}
\mathcal{P}^{W}_C(\x):=\{P\in\mathcal{P}(\Xi,\mathcal{F})\;|\;\mathcal{W}(P,P_0)\le r(\x)\}.
\end{equation}
The reformulation of \eqref{opt:DRO} with the ambiguity set $\mathcal{P}^{W}_C(\x)$ is given by the following theorem.
\begin{theorem} 
\label{thm:reform-D3RO-Wass}
The \eqref{opt:DRO} with the ambiguity set $\mathcal{P}^{W}_C(\x)$ can be reformulated as the following semi-infinite program:
\begin{equation}\label{opt:reform-DRO-wass-cont}
\begin{aligned}
 \underset{\bs{x},\bs{v}}{\emph{min}} &\quad f(\bs{x}) + \frac{1}{N}\sum^{N}_{i=1} v_i + r(\bs{x})\cdot v_{N+1} & \\ 
\emph{s.t.} &\quad  h(\bs{x},\bs{s}) - v_i - v_{N+1}\cdot d(\bs{s},\bs{\xi}^i)\le 0, & \forall \bs{s}\in\Xi, \forall i\in[N]  \\    
&\quad \bs{x}\in \mathcal{X}, \;\;  v_1,\ldots, v_{N}\in \mathbb{R}, \quad v_{N+1} \ge 0. &     
\end{aligned}
\end{equation}
\end{theorem}
\begin{proof}
From Theorem~3.6 of \citep{luo2017-decomp-alg-distr-rob-opt}, the inner problem of \ref{opt:DRO} with the ambiguity set defined by the Wasserstein metric \eqref{def:wass-metric-cont} is equivalent to the following conic linear program:
\begin{equation}\label{opt:inner-reform-wass-cont}
\begin{aligned}
 \underset{\mu}{\textrm{max}} &\quad   \int_{\bs{s}\in\Xi} h(\bs{x},\bs{s}) \mu(d\bs{s}\times\Xi) &  \\ 
\textrm{s.t.} &\quad \mu(\Xi\times \{\bs{\xi}^i\})  = 1/N, & \forall i\in [N]  \\
&\quad \mu(\Xi\times\Xi^{\prime}) = 0, \\
&\quad \sum_{i\in[N]} \int_{\bs{s}\in\Xi}  d(\bs{s},\bs{\xi}^i) \mu(d\bs{s}\times\{\bs{\xi}^i\})\le r(\bs{x}), & \\
&\quad  \mu \succeq 0, &   
\end{aligned}
\end{equation}   
where $\Xi^{\prime}:=\Xi\setminus\{\bs{\xi}^i\}^{N}$, and $\mu\succeq 0$ denotes that $\mu$ is a positive measure.
Based on Theorem~3.7 of \citep{luo2017-decomp-alg-distr-rob-opt}, 
we can apply the conic duality theory from \citep{shapiro2001_conic-lp} to \eqref{opt:inner-reform-wass-cont}, 
and obtain the following dual formulation of \eqref{opt:inner-reform-wass-cont}:
\begin{equation}\label{opt:inner-dual-wass-cont}
\begin{aligned}
 \underset{\bs{v}}{\textrm{min}} &\quad \displaystyle \frac{1}{N}\sum^{N}_{i=1} v_i + r(\bs{x})\cdot v_{N+1} & \\ 
\textrm{s.t.} &\quad h(\bs{x},\bs{s}) - v_i - v_{N+1}\cdot d(\bs{s},\bs{\xi}^i)\le 0 & \forall \bs{s}\in\Xi, \forall i\in[N],  \\    
&\quad  v_1,\ldots, v_{N}\in \mathbb{R}, \quad v_{N+1} \ge 0. &     
\end{aligned}
\end{equation}
After combining \eqref{opt:inner-dual-wass-cont} with the outer minimization problem over $x$, we obtain the desired reformulation \eqref{opt:reform-DRO-wass-cont}.
\end{proof}

\section{Reformulation of \eqref{opt:DRO} with K-S distance and continuous support of the random parameters}
\label{appendix-KS}
We now investigate the reformulation of \eqref{opt:DRO} with the ambiguity set defined by K-S distance 
where the support $\Xi$ is not finite. We assume that $\Xi$ is contained in a hyper-rectangle
$[\bs{a},\bs{b}]:=[a_1,b_1]\times\ldots\times[a_d,b_d]$. The definition of K-S distance \eqref{def:KS-univariate} can be 
generalized for two multivariate cumulative distribution functions as follows:
\begin{equation}
D(P_1,P_2)=\underset{\bs{s}\in\mathbb{R}^d}{\textrm{sup}}\; |F_1(\bs{s})-F_2(\bs{s})|,
\end{equation}  
where $\bs{s}=[s_1,\ldots,s_d]$ and the cumulative function $F_i$ ($i=1,2$) is defined as: 
$F_i(\bs{s})=\int^{s_1}_{-\infty}\dots\int^{s_d}_{-\infty}P_1(x_1,\ldots,x_d)dx_1\ldots dx_d$. 
Suppose for a $\x_0\in X$, we have observed $N$ samples of the random vector $\bs{\xi}$,
written as $\{\bs{\xi}^i\}^N_{i=1}$. We define the empirical distribution as $P_0:=\sum^N_{i=1}\frac{1}{N}\delta_{\bs{\xi}_i}$
and denote the cumulative distribution function of $P_0$ as $F_0$. Let $\mathcal{P}([\bs{a},\bs{b}],\mathcal{B})$ denote the set of 
probability distributions on $[a,b]$ with the Borel sigma algebra $\mathcal{B}$. For any $P\in\mathcal{P}([\bs{a},\bs{b}],\mathcal{B})$,
let $F^P$ denote the cumulative distribution function of $P$. The decision dependent ambiguity set based on K-S distance
can be constructed as:
\begin{equation}\label{eqn:P_KS_cont}
\mathcal{P}^{KS}_C(\x):=\big\{P\in\mathcal{P}([\bs{a},\bs{b}],\mathcal{B}):\; \underset{\bs{s}\in[\bs{a},\bs{b}]}{\textrm{sup}}\;|F^P(\bs{s})-F_0(\bs{s})|\le\alpha(\x) \big\}.
\end{equation}
We now reformulate the ambiguity set \eqref{eqn:P_KS_cont} into finitely many expectation constraints of indicator functions 
by partitioning the hyper-rectangle $[\bs{a},\bs{b}]$ into hyper-rectangular cells. 
Specifically, let $\xi^i_k$ be the $k$-th component of the $i$-th observed sample. 
For each component $k$ ($k\in[d]$), we sort the observed samples $\{\xi^i_k\}^N_{i=1}$
based on the $k$-th component in the ascending order, and suppose the sorted sample
components are labeled as $\{\xi^{[i]}_k\}^N_{i=1}$ such that $a_k<\xi^{[1]}_k<\xi^{[2]}_k<\dots<\xi^{[N]}_k<b_k$. 
Let us divide each interval $[a_k, b_k]$ into $N+1$ sub-intervals as: $I^k_0=[a_k, \xi^{[1]}_k)$, 
$I^k_i=[\xi^{[i]}_k, \xi^{[i+1]}_k)$ for $i\in[N-1]$ and $I^k_{N}=[\xi^{[N]}_k, b_k]$, and
create an $N$-dimensional grid based on the sub-intervals for each dimension $k$ to partition $[\bs{a},\bs{b}]$ into $(N+1)^d$ sub-rectangular cells. 
Based on this partition and using the convention that $\xi^{[0]}_k=a_k$ for $k\in[d]$, the reference CDF can be written as:
\begin{equation}
F_0(\bs{s})=\frac{N_{j_1j_2\dots j_d}}{N}\qquad \textrm{for } \bs{s}\in I^1_{j_1}\times I^2_{j_2}\times \dots \times I^d_{j_d}, \quad j_r\in\{0,1,\dots,N\},\quad r\in [d],
\end{equation}
where $N_{j_1j_2\dots j_d}$ is the number of observed samples
within the hyper-rectangle $[a_1, \xi^{[j_1]}_1]\times[a_2, \xi^{[j_2]}_2]\times\dots\times[a_d, \xi^{[j_d]}_d]$. 
For simplicity of notations, we let $I_{j_1j_2\dots j_d}:=I^1_{j_1}\times I^2_{j_2}\times \dots \times I^d_{j_d}$,
then the ambiguity set \eqref{eqn:P_KS_cont} can be reformulated as
\begin{equation}\label{eqn:P-KS-C-reform}
\mathcal{P}^{KS}_C(\x)=\left\{P\in\mathcal{P}([\bs{a},\bs{b}],\mathcal{B}):\; \left|P(\bs{s}\in I_{j_1j_2\dots j_d}) - \frac{N_{j_1\dots j_d}}{N} \right|\le \alpha(\x)\textrm{ for } 
 j_r\in\{0,1,\dots,N\},\; r\in [d] \right\}.
\end{equation}
Reformulation of the \eqref{opt:DRO} with the ambiguity set \eqref{eqn:P_KS_cont} is given in the following theorem:
\begin{theorem}\label{thm:reform-D3RO-KS}
If $h(\x,\bs{s})$ is continuous in $\bs{s}\in[\bs{a},\bs{b}]$ for any $\bs{x}\in X$, the \eqref{opt:DRO} with the ambiguity set 
\eqref{eqn:P_KS_cont} can be reformulated as the following semi-infinite program:
\begin{equation}\label{opt:DRO-KS-cont-dual}
\begin{aligned}
&\underset{\bs{x},\bs{\overline{\lambda}},\bs{\underline{\lambda}}}{\emph{min}}\;\; f(\x) + \sum^N_{j_1=0}\dots\sum^N_{j_d=0}\left(\frac{N_{j_1\dots j_d}}{N} + \alpha(\x) \right)\overline{\lambda}_{j_1\dots j_d} + \sum^N_{j_1=0}\dots\sum^N_{j_d=0}\left(\frac{N_{j_1\dots j_d}}{N} - \alpha(\x) \right)\underline{\lambda}_{j_1\dots j_d} + \gamma \\
&\emph{ s.t. }\;\; \overline{\lambda}_{j_1\dots j_d} + \underline{\lambda}_{j_1\dots j_d} + \gamma \ge h(\x,\bs{s}) \qquad \forall \bs{s}\in \emph{cl}(I_{j_1\dots j_d}),\;\forall j_r\in\{0,1,\dots,N\},\; r\in[d], \\
&\qquad \gamma\in\mathbb{R},\;\overline{\lambda}_{j_1\dots j_d}\le 0,\; \underline{\lambda}_{j_1\dots j_d}(\bs{s})\ge 0,\;\; \forall j_r\in\{0,1,\dots,N\},\; r\in[d].
\end{aligned}
\end{equation}
\end{theorem}
\begin{proof}
Note that the probability $P(\bs{s}\in I_{j_1j_2\dots j_d})$ in \eqref{eqn:P-KS-C-reform} can be written as the expectation 
of the indicator function $\bs{1}_{I_{j_1j_2\dots j_d}}(\bs{s})$ with respect to $P$. The inner problem of \eqref{opt:DRO}
can be reformulated as the following conic linear program:
\begin{equation}\label{opt:KS-cont-inner}
\begin{aligned}
&\underset{P\in\mathcal{P}([\bs{a},\bs{b}],\mathcal{B})}{\textrm{max}}\;\; \E_P[h(\x,\bs{s})]  & \\
&\qquad \textrm{ s.t. }\quad \E_P[\bs{1}_{I_{j_1\dots j_d}}(\bs{s})] \le \frac{N_{j_1\dots j_d}}{N} + \alpha(\x) \quad \forall j_r\in\{0,1,\dots,N\},\;\; r\in[d], & \overline{\lambda}_{j_1\dots j_d}\le 0 \\
&\qquad\qquad\quad  \E_P[\bs{1}_{I_{j_1\dots j_d}}(\bs{s})] \ge \frac{N_{j_1\dots j_d}}{N} - \alpha(\x) \quad \forall j_r\in\{0,1,\dots,N\},\;\; r\in[d], &\underline{\lambda}_{j_1\dots j_d}\ge 0 \\
&\qquad\qquad\quad \E_P[\bs{1}_{[\bs{a},\bs{b}]}(\bs{s})] = 1,  & \gamma\in\mathbb{R}.
\end{aligned}
\end{equation}
Applying conic duality \citep{shapiro2001_conic-lp} to \eqref{opt:KS-cont-inner} we obtain the following dual problem of \eqref{opt:KS-cont-inner}:
\begin{equation}\label{opt:KS-cont-inner-dual}
\begin{aligned}
&\underset{\bs{\overline{\lambda}},\bs{\underline{\lambda}},\gamma}{\textrm{min}}\;\; \sum^N_{j_1=0}\dots\sum^N_{j_d=0}\left(\frac{N_{j_1\dots j_d}}{N} + \alpha(\x) \right)\overline{\lambda}_{j_1\dots j_d} + \sum^N_{j_1=0}\dots\sum^N_{j_d=0}\left(\frac{N_{j_1\dots j_d}}{N} - \alpha(\x) \right)\underline{\lambda}_{j_1\dots j_d} + \gamma \\
&\textrm{ s.t. }\quad (\overline{\lambda}_{j_1\dots j_d} +  \underline{\lambda}_{j_1\dots j_d} )\bs{1}_{I_{j_1\dots j_d}}(\bs{s}) + \gamma \bs{1}_{[\bs{a},\bs{b}]} \ge h(\x,\bs{s}) \qquad \forall \bs{s}\in [\bs{a},\bs{b}],  \\
&\qquad\quad \gamma\in\mathbb{R},\;\overline{\lambda}_{j_1\dots j_d}\le 0,\; \underline{\lambda}_{j_1\dots j_d}(\bs{s})\ge 0,\;\; \forall j_r\in\{0,1,\dots,N\},\; r\in[d].
\end{aligned}
\end{equation}
Note that by partitioning the range of the vector $\bs{s}$, the first constraint of \eqref{opt:KS-cont-inner-dual} 
can be reformulated as the following semi-infinite constraints: 
\begin{equation}
\overline{\lambda}_{j_1\dots j_d} + \underline{\lambda}_{j_1\dots j_d} + \gamma \ge h(\x,\bs{s}) \qquad \forall \bs{s}\in I_{j_1\dots j_d},\;\forall j_r\in\{0,1,\dots,N\},\; r\in[d].
\end{equation}
Since $h(\x,\bs{s})$ is continuous in $\bs{s}\in[\bs{a},\bs{b}]$ for any $\x\in X$, we can replace $I_{j_1\dots j_d}$ 
with the closure $\textrm{cl}(I_{j_1\dots j_d})$ in the above semi-infinite constraints. Then by Proposition~2.8(iii) of \citep{shapiro2001_conic-lp},
the optimal objective of \eqref{opt:KS-cont-inner} equals the optimal objective of \eqref{opt:KS-cont-inner-dual}. 
After combining \eqref{opt:KS-cont-inner-dual} with the outer minimization problem over $\x\in X$, we can reformulate 
\eqref{opt:DRO} into \eqref{opt:DRO-KS-cont-dual}.
\end{proof}

\section{A cutting-surface algorithm for solving the reformulation of \eqref{opt:DRO} with a continuous support of the random parameters}
\label{appendix-cutting-surface-algorithm}
We see from Theorems~\ref{thm:reform-D3RO-Wass} and \ref{thm:reform-D3RO-KS}  that
under a continuous support $\Xi$, the reformulations of \eqref{opt:DRO} are special cases of a semi-infinite program.
Let us consider the following general form of a semi-infinite program:
\begin{equation}
\begin{array}{cll}
 \underset{x}{\textrm{min}} & \displaystyle f(x) & \\ 
\textrm{s.t.} &\displaystyle g(x,t)\le 0, & \forall t\in T, \\    
& x\in X,    
\end{array}
\tag{gen-SIP}\label{opt:gen-SIP}
\end{equation}
where $X\subseteq\mathbb{R}^{k_1}$ and $T\subseteq\mathbb{R}^{k_2}\times\mathbb{Z}^{k_3}$, 
allowing that $T$ may be defined as a mixed-integer set. 
The cutting-surface algorithm is given in Algorithm~\ref{alg:gen-SIP}. 
The idea of the cutting-surface algorithm is to solve a relaxation problem (or a master problem) of the semi-infinite program at each iteration, 
where the relaxation problem has a finite number of constraints. 
An additional constraint that is violated by the solution of the current relaxation problem is added 
to the current set of constraints for the relaxation problem in the next iteration.  
Algorithm~\ref{alg:gen-SIP} is based on an oracle to solve the master problem
\begin{equation}\label{eqn:master-prob}
 \underset{x\in X}{\textrm{min}}\;\{ f(x):\; \textrm{s.t. } g(x,t)\le 0, \; t\in T^{\prime}   \}, 
\end{equation}
where $T^{\prime}$ is a finite subset of $T$, and an oracle to solve the separation problem
\begin{equation}\label{eqn:sep-prob}
\underset{t\in T}{\textrm{max}}\; g(\hat{x}, t),
\end{equation}
for any $\hat{x}\in X$. It outputs an $\varepsilon$-optimal solution to \eqref{opt:gen-SIP}, 
where the accuracy of a solution to \eqref{opt:gen-SIP} is defined in Definition~\ref{def:accuracy-gen-SIP}. 
Theorem~\ref{thm:e-suboptimal-WRO} shows that Algorithm~\ref{alg:gen-SIP} terminates in finitely many iterations
if $X\times T$ is compact and $g(x,t)$ is continuous on $X\times T$.  
\begin{definition}
\label{def:accuracy-gen-SIP}
For a general semi-infinite program in the form of \eqref{opt:gen-SIP}, 
a point $x_0\in X$ is an $\varepsilon$-feasible solution of \eqref{opt:gen-SIP} if $\underset{t\in T}{\textrm{max}}\; g(x_0,t)\le \varepsilon$.
A point $x_0\in X$ is an $\varepsilon$-optimal solution of \eqref{opt:gen-SIP} if $x_0$ is an $\varepsilon$-feasible solution of \eqref{opt:gen-SIP} and
$f(x_0)\le\textrm{Val}\eqref{opt:gen-SIP}$.
\end{definition}

{\centering
\begin{minipage}{0.95\linewidth}
\begin{algorithm}[H] 
	\caption{A cutting-surface algorithm (modified exchange algorithm) to solve \eqref{opt:gen-SIP}. }
	\label{alg:gen-SIP}
	\begin{algorithmic}
	\State {\bf Prerequisites}: An oracle that generates the optimal solution to the master problem \eqref{eqn:master-prob} 
							and an oracle that generates an $\varepsilon$-optimal solution to the separation problem \eqref{eqn:sep-prob}.
	\State {\bf Output}: An $\varepsilon$-optimal solution of (\ref{opt:gen-SIP}).
	 \State{\bf Step 1} Set $T_0\gets\emptyset$, $k\gets 0$.   
	 \State{\bf Step 2}  Determine an optimal solution $x_k$ of the problem 
	 				  $\underset{x\in X}{\textrm{min}}\;\{ f(x):\; \textrm{s.t. } g(x,t)\le 0, \; t\in T_k   \}$.
   
	 \State{\bf Step 3}  Determine a $\frac{\varepsilon}{2}$-optimal solution $t_{k+1}$ of the problem $\underset{t\in T}{\textrm{max}}\; g(x_k, t)$.
	 				  If $g(x_k,t_{k+1})\le \frac{\varepsilon}{2}$, stop and return $x_k$; otherwise let $T_{k+1} \gets T_{k}\cup\{ t_{k+1}\}$, 
					  $k\gets k+1$ and go to Step 2	 
     
       	\end{algorithmic}	 
\end{algorithm}
\end{minipage}
\par
}

\begin{theorem}[Theorem~7.2 in \citep{hettich1993_SIP-thy-methd-appl}]\label{thm:e-suboptimal-WRO}
If $X\times T$ is compact, and $g(x,t)$ is continuous on $X\times T$, 
then Algorithm~\ref{alg:gen-SIP} terminates in finitely many iterations and 
returns an $\varepsilon$-optimal solution of \eqref{opt:gen-SIP}. 
\end{theorem}
 
We note that the oracle problem \eqref{eqn:sep-prob} in the cutting-surface algorithm is simply a function evaluation
problem for the decision dependent but finite support case. Therefore, in this case the algorithm can be adapted by sequentially adding 
cuts as constraints based on violated inequalities are identified.

\end{appendices}

% Acknowledgments here
\section*{Acknowledgement}
This research was supported by the Office of Naval Research grant N00014-18-1-2097-P00001.

\bibliography{reference-database-07-03-2017}

%%%%%%%%%%%%%%%%%
\end{document}